\documentclass{rsepublic}
\Year{2016}
\Volume{xx}
\Paper{xx}
\Pagerange{\pageref{firstpage}--\pageref{lastpage}}
\usepackage{amsmath}
  \usepackage{paralist}
  \usepackage{graphics} 
  \usepackage{epsfig} 
\usepackage{graphicx}  \usepackage{epstopdf}
 \usepackage[colorlinks=true]{hyperref}
\hypersetup{urlcolor=blue, citecolor=red}
\usepackage{hyperref}
\usepackage{marginnote}

\theoremstyle{definition}

\newtheorem*{remark}{Remark}

\newcommand{\brackets}[1]{\left( #1 \right)}
\newcommand{\ip}[2]{\left\langle #1, #2 \right \rangle}
\newcommand{\xhat}{\hat{{\bf x}}}
\newcommand{\yhat}{\hat{{\bf y}}}
\newcommand{\zhat}{\hat{{\bf z}}}
\newcommand{\R}{\mathbb{R}}

\newcommand{\Dop}[0]{\mathcal{D}}
\newcommand{\norm}[1]{\left\| #1  \right\|} 
\renewcommand{\d}[1]{\,  \mathrm{d} #1} % for integrals
\newcommand{\abs}[1]{\left| #1 \right|} % for absolute value
\renewcommand{\v}[1]{\ensuremath{\mathbf{#1}}} % for vectors
\newcommand{\half}[0]{\tfrac{1}{2}}
\newcommand{\sech}{\operatorname{sech}}

\newcommand{\uv}[1]{\ensuremath{\hat{\mathbf{#1}}}}
\newcommand{\beq}[0]{\begin{equation}}
\newcommand{\eeq}[0]{\end{equation}}
\newcommand{\pd}[2]{\frac{\partial #1}{\partial #2}} 
\newcommand{\pdd}[2]{\frac{\partial^2 #1}{\partial #2^2}}
\newcommand{\avg}[1]{\left< #1 \right>} % for average
\newcommand{\e}{\operatorname{e}}
\newcommand{\weak}{\rightharpoonup}

\newcommand{\ddbydd}[2]{\frac{\mathrm{d}^2 #1}{\mathrm{d} #2^2}} % for double derivatives

\newcommand{\pola}{\psi}
\newcommand{\azim}{\beta}

\begin{document}
\title[Travelling waves in ferromagnetic nanowires]
      {Existence of travelling-wave solutions representing domain wall motion in a thin ferromagnetic nanowire}
\author[R. G. Lund, J. M. Robbins and V. Slastikov]{\textbf{Ross G. Lund$^1$, J. M. Robbins$^2$ and  Valeriy Slastikov$^2$}
\\
$^1$Department of Mathematical Sciences, New Jersey Institute of Technology, University Heights, Newark, NJ 07102, USA. \\$^2$University of Bristol School of Mathematics, University Walk, Bristol BS8 1TW, UK.}
\MSdates[?Received date?]{?Accepted date?}
\label{firstpage}
\maketitle

\begin{abstract}
We study the dynamics of a domain wall under the influence of applied magnetic fields in a one-dimensional ferromagnetic nanowire, governed by the Landau--Lifshitz--Gilbert equation. Existence of travelling-wave solutions close to two known static solutions is proven using implicit-function-theorem-type arguments. 
\end{abstract}
%%%%%%%%%%%%%%%%%%%%%%%%%%%%%%%%%%%%%%%%%%%%%%%%%%%%%%%%%%%%%

\section{Introduction}

The subject of domain wall motion in ferromagnetic nanowires has attracted a great deal of interest in recent years, from both the physics and applied mathematics communities. This is, in part, due to proposed magnetic storage media such as the so-called racetrack memory \cite{Parkin}. This device makes use of the fact that, in a thin ferromagnetic nanowire, the magnetic domains (regions of uniform magnetization, which are separated by small transition layers called \emph{domain walls}) prefer to align along the wire axis (in either direction), thus providing a two-state system which can be used to represent information. In order to read or write information, the domains are propagated through the wire by application of a magnetic field or an electric current.

In order to understand the workings of such a device, the primary problem to study is the dynamics of a single domain wall separating two domains of opposite magnetization, under the influence of an applied magnetic field. The magnetization dynamics is modelled by the Landau--Lifshitz--Gilbert (LLG) equation \cite{LL, Gilbert}. 

 An ansatz method that has been successful in approximately describing one-dimensional domain wall motion  was introduced by Schryer and Walker \cite{Walker},  generalized by Malozemoff and Slonczewski \cite{Slonczewski},  and applied to various cases  by a number of authors (see, for example, \cite{Thiaville05, Mougin07, Bryan08, Yang08, Wang09a, Wang09b, Lu10, Tretiakov10, Tretiakov12}). This approach consists of looking for approximate solutions belonging to a three-parameter family of profiles, in which the parameters represent translations, rotations and re-scalings of a suitably chosen static profile.  The time dependence of the parameters is determined so as to satisfy the equations of motion as nearly as possible. A recent asymptotic analysis \cite{PRSA} provides  a systematic foundation for this approach.

The phenomenology of solution behaviours in this problem is reasonably well understood by applying the approximate methods discussed above. Exact solutions are obtained only in some special cases \cite{Walker, Prec}, including a travelling-wave (TW) solution known as the Walker solution. 
For values of the driving field (that is, the component of the applied field along the wire axis) up to a certain critical value, which depends on other physical parameters (such as the transverse applied field components and the anisotropy coefficients of the nanowire), the velocity and width of the domain wall are asymptotically constant in time: the behaviour is apparently that of a TW. For fields above the critical value, oscillatory motion sets in: the domain wall precesses about the wire axis, and its velocity and width oscillate in time. This phenomenon is known in the literature as Walker breakdown, due to its discovery in the specific case studied by Schryer and Walker \cite{Walker}.  An approximate description of these transitions in the general 1D problem, including predictions of the critical fields, has been provided elsewhere \cite{PRSA, PRB}.

Notwithstanding the success of these approximate methods, rigorous results concerning domain wall motion in nanowires are currently few \cite{Carbou1, Carbou2}.  In this article, we analyze TW solutions of the LLG equation in thin nanowires. We consider nanowires of biaxial anisotropy, with easy axis along the wire, subject to a uniform magnetic field with components along and transverse to the wire, and prove that TW solutions exist provided that the magnetic field component along the wire is sufficiently small.  The proof is an application  of the Implicit Function Theorem.
 
 \subsection{Domain wall motion in a thin nanowire: statement of the problem}

We work in the continuum-mechanical theory of micromagnetics (see e.g.~\cite{Hubert}). The equation governing magnetization dynamics is the LLG equation. We study the following dimensionless form:
\beq
 {\uv{M}}_t + \alpha \uv{M} \times {\uv{M}}_t=%(1+\alpha^2)
\uv{M}\times \v{H}(\uv{M}),
\label{LLG}
\eeq
where $\uv{M} \in S^2$ is the magnetization, $\alpha > 0$ is the Gilbert damping constant (typical value in the range 0.01--0.2), and $\v{H}(\uv{M})$ is the effective magnetic field.   We study the LLG equation in a \emph{thin} cylindrical magnetic nanowire, $\Omega \subset \R^3$; `thin' meaning that the radius of the wire is small in comparison to both the length of the wire and the magnetic exchange length of the material. We work in scaled units such that the exchange length is equal to 1, and take the length of the wire to be infinite. We use Cartesian coordinates on $\Omega$: $x$ for the coordinate along the wire axis, and $y,z$ for the coordinates on the cross-section.

 In the thin-wire regime it is known that the micromagnetic energy $\Gamma$-converges to a one-dimensional energy \cite{Kuhn, Slasberg}, given in dimensionless form by
\beq
E(\uv{M}) = \frac{1}{2} \int_\R |\uv{M}_x|^2 \d x + \frac12 \int_\R \left( 1 - (\uv{M}\cdot \xhat)^2 + K_2(\uv{M}\cdot \yhat)^2 \right) \d x,
\label{Energy}
\eeq
where $ K_2\geq 0$ is the effective hard-axis anisotropy coefficient and $\xhat, \yhat$ and $\zhat$ are the usual Cartesian basis vectors. %Also, 
The magnetization vector depends only on the axial coordinate, $x\in \R$, and time, $t\in \R^+$, i.e. $\uv{M}=\uv{M}(x,t)$. 
The effective field is given by
\beq
\v{H}(\uv{M})=-\frac{\delta E}{\delta \uv{M}} + \v{H}_a = \uv{M}_{xx} + (\uv{M}\cdot \xhat)\xhat - K_2(\uv{M}\cdot \yhat)\yhat + \v{H}_a,
\label{EffectiveField}
\eeq
where $\v{H}_a=H_1 \xhat +H_2 \yhat +H_3 \zhat \in \R^3$ is a uniform applied magnetic field. For convenience, we denote the parameters of the system, namely the hard-axis anisotropy $K_2$ and applied magnetic field $\v{H}_a$, collectively by 
\beq
\Lambda = (H_1,H_2,H_3,K_2).
\eeq
We note that the effective field may also be written as 
\beq
\v{H}(\uv{M}) = \uv{M}_{xx} - \nabla_{\uv{M}} U^\Lambda(\uv{M}),
\eeq
where the potential $U^\Lambda$ is given by
\beq
U^\Lambda(\uv{M})= \frac12 \brackets{1 - (\uv{M}\cdot \xhat)^2 + K_2(\uv{M}\cdot \yhat)^2 - 2 \v{H}_a \cdot \uv{M}}.
\label{Potential}
\eeq

As we are interested in domain wall dynamics, we impose the boundary conditions
\beq
\lim_{x \to \pm \infty} \uv{M} = \uv{m}^\Lambda_\pm,
\label{BC}
\eeq
where $\uv{m}^\Lambda_\pm$ correspond to distinct local minima of $U^\Lambda$  whose projections along $\xhat$ have opposite signs.
For definiteness, we take
\beq
\xhat\cdot\uv{m}^\Lambda_+  > 0, \quad \xhat\cdot\uv{m}^\Lambda_- < 0,
\label{tailtotail}
\eeq
which corresponds to `tail-to-tail' domains (`head-to-head' domains may be treated similarly).

It will be convenient to introduce polar coordinates  $(\pola(\xi), \azim(\xi))$ for the magnetization, writing %$\uv{m}$ a
\beq
\uv{m}(\pola, \azim) = (\sin\pola \cos\azim, \cos\pola, \sin\pola \sin\azim),
\label{PolarCoordinates}
\eeq
and letting $(\pola^\Lambda_\pm,  \azim^\Lambda_\pm)$ denote the polar coordinates of the boundary values $\uv{m}^\Lambda_\pm$. 
We note that we have taken the polar axis along the hard axis $\yhat$ rather than,  as is often done, the easy axis $\xhat$.
This is because the  profiles of interest do not take values near the hard axis, which is energetically costly, so that their  representation in terms of $(\pola, \azim)$ avoids the coordinate singularities at  $\pola = 0$ and $\pola = \pi$.

%%%%%%%%%%

\subsection{Overview}
The remainder of the paper is structured as follows: In \S2 we derive the system of ODEs  satisfied by TWs and identify static (zero velocity) solutions for two subsets of parameters in which the component of the applied field along the wire vanishes, namely 
 i)  $K_2 >0, \v{H}_a = 0$, corresponding to biaxial anisotropy and  vanishing applied field;  and ii) $K_2 = H_1=0 < H_2^2 +  H_3^2 <1$, corresponding to uniaxial anisotropy and nonvanishing transverse applied field.
In \S3 we reformulate the TW equations as a map between Banach spaces whose zeros correspond to solutions of the TW equations. 
In \S4 we state and prove our main results, which establish existence of TW solutions to the LLG equation \eqref{LLG} for values of the parameters in neighbourhoods of the  subsets i) and ii).

%%%%%%%%%%%%%%%%%%%%%%%%%%%%%%%%%%%%%%%%%%%%%%%%%%%%%%%

\section{Travelling Waves.}
We look for TW solutions of the LLG equation \eqref{LLG} of the form
\beq
\uv{M}(x,t) = \uv{m}(x-Vt),
\label{Ansatz}
\eeq
where $V$ denotes the velocity of the TW.  Substituting \eqref{Ansatz} into \eqref{LLG} and imposing the boundary conditions \eqref{BC}, we obtain the TW equation
\beq
V\uv{m}'+\alpha V \uv{m}\times \uv{m}' + \uv{m} \times \v{H}(\uv{m})=0, \quad \lim_{\xi \to \pm  \infty}{\uv{m}}(\xi) = \uv{m}^\Lambda_{\pm},
\label{TWEq}
\eeq
a system of second-order ODEs, where we have introduced the travelling coordinate $\xi:=x-Vt$.  We note that its solutions  are determined up to translation; that is, if $\uv{m}(\xi)$ satisfies \eqref{TWEq}, so does $\uv{m}(\xi - \xi_0)$.

%%%%%%%%%%%%%%%%%%%%%%%%%

\subsection{Static solutions.} 

Given a solution $\uv{m}$ of the TW equation \eqref{TWEq}, one can derive the following identity for the velocity \cite{PRB}:
\beq
\label{V_identity}
 V = \frac{U^\Lambda(\uv{m}^\Lambda_-) - U^\Lambda(\uv{m}^\Lambda_+)} {\alpha \int_\R \uv{m}' \cdot \uv{m}' }.
\eeq
(Formally,  \eqref{V_identity} is obtained by noting that the TW equation \eqref{TWEq} implies that $\alpha V \uv{m}' \cdot \uv{m}'$ is equal to $\left( \half  \uv{m}' \cdot \uv{m}' - 
U^\Lambda(\uv{m})\right)'$.)
Therefore, if the stable domain orientations $\uv{m}^\Lambda_+$ and $\uv{m}^\Lambda_-$ have same energy, then
the velocity  vanishes -- equivalently, there may exist a static solution of the TW equation.  We always have that $U^\Lambda(\uv{m}^\Lambda_- ) =U^\Lambda(\uv{m}^\Lambda_+)$ whenever $H_1 = 0$, i.e.~the component of the applied field along the wire vanishes, since in this case, the potential $U^\Lambda(\uv{m})$ is symmetric under the reflection $(m_1,m_2,m_3) \rightarrow (-m_1, m_2, m_3)$.

We can find explicit static solutions in two parameter regimes with $H_1 = 0$.
The first is the case of vanishing transverse applied field, which we denote by
\beq
\Lambda_W=(0,0,0,K_2), \quad K_2 > 0.
\label{Anisotropy}
\eeq
Allowing $H_1 \neq 0$, this encompasses the regime of the Walker solution. We note that, in this regime, we have restricted to positive $K_2$ since the case of $K_2=H_2=H_3=0$ is degenerate and our results do not apply there; though an explicit dynamic solution to \eqref{LLG} is available which is not a simple travelling wave \cite{Prec}. 

The static profile $\uv{m}_W$, for $\Lambda=\Lambda_W$, is given by the usual Bloch domain wall, 
\beq
\uv{m}_W(\xi) = (\tanh \xi, 0, \sech \xi).
\label{StaticStar}
\eeq
The polar representation of the static profile \eqref{StaticStar} is given by
\beq
\pola_W(\xi) = \frac{\pi}{2},\quad \azim_W(\xi) = 2\tan^{-1}(\e^{-\xi}). 
\eeq
It is easily shown that $\uv{m}_W$ corresponds to a minimizer of the micromagnetic energy \eqref{Energy}. 
A second minimizer, $(\tanh \xi, 0, - \sech \xi)$, may be obtained by reflection through the intermediate axis $\zhat$.  There exist another pair of solutions of \eqref{TWEq} given by $(\tanh \xi, \pm \sech \xi, 0)$, which correspond to critical points 
of the energy, in which the domain wall is aligned along the hard axis $\yhat$.  As they are energetically unstable, we do not consider them further.  

The second regime corresponds to vanishing hard-axis anisotropy and nonvanishing transverse field, and is denoted
\beq
\Lambda_T=(0,H_2,H_3,0), \quad 0<H_2^2+H_3^2<1.
\label{TransverseField}
\eeq
In view of the rotational symmetry about the easy axis, we may assume $H_2=0$. In terms of the polar coordinates $(\pola,\azim)$,  
the static profile is given by the solution of the first-order equation
\beq
\pola_T = \frac{\pi}{2}, \quad \azim_T' = H_3 - \sin \azim_T,
\label{StaticZero}
\eeq
with boundary conditions
\beq \label{eq: StaticZero bcs}
\lim_{\xi\to -\infty} \beta_T = \pi - \sin^{-1}H_3, \quad \lim_{\xi\to \infty} \beta_T = \sin^{-1}H_3.
\eeq
One may explicitly compute 
$\azim_T$ 
(see \cite{PRB}), but the expression is rather long and in any case will not be required in what follows.   We will only need the fact that $\sin \azim_T \geq H_3$ so that $ \azim_T $ is strictly decreasing.

%%%%%%%%%%%%%%%%%%%%%%%%%%

\section{Travelling-wave equations as a map between Banach spaces}

Let $\Lambda_*$ denote either of the static parameters $\Lambda_W$ or $\Lambda_T$, and let $(\pola_*,\azim_*)$ denote the corresponding static solutions. We introduce reference profiles $(\pola^\Lambda, \azim^\Lambda)$ that coincide with  $(\pola_*,\azim_*)$ for $\Lambda = \Lambda_*$ as follows: 
\begin{align}
\pola^\Lambda(\xi) &:= \pola_*(\xi) + \Theta(\xi) (\pola_+^\Lambda - \pola_+^{\Lambda_*}) + \Theta(-\xi) (\pola_-^\Lambda - \pola_-^{\Lambda_*}),\nonumber\\
\azim^\Lambda(\xi) &:= \beta_*(\xi) + \Theta(\xi) (\beta_+^\Lambda - \beta_+^{\Lambda_*}) + \Theta(-\xi) (\beta_-^\Lambda - \beta_-^{\Lambda_*}).
\label{ReferenceConfiguration}
\end{align}
Here, $\Theta(\xi)$ is a switching function with the following properties: i) $\Theta(\xi) = 0$  for $\xi < 0$, ii)  $\Theta(\xi) = 1$  for $\xi >  \xi_0$ for some $\xi_0 > 0$, and iii) on the interval $(0,\infty)$, $1-\Theta$ belongs to  $H^2((0,\infty))$. It is then straightforward to check that $\lim_{\xi\pm \infty} (\pola^\Lambda(\xi), \azim^\Lambda(\xi)) =
(\pola^\Lambda_\pm, \azim^\Lambda_\pm)$ and that ${\pola^\Lambda}'', {\azim^\Lambda}'' \in L^2(\R)$.

We consider profiles with polar-coordinate representation $(a,b)$ of the form
\beq a = \psi^\Lambda + u, \quad b= \beta^\Lambda + w,\label{DefineAB}\eeq 
where $u,w  \in H^2(\R)$. 
These satisfy the boundary conditions \eqref{BC} and include all profiles sufficiently close to the reference profile $(\psi^\Lambda,  \beta^\Lambda )$.  Substituting $\uv{m}(a,b) := (\sin a \cos b, \cos a , \sin a \sin b)$ 
into 
the TW equation \eqref{TWEq}, we obtain a  second-order system for $u$ and $w$ in the form $G_{1} = G_2 = 0$, where
\begin{align}
G_{1}(u,w,V;\Lambda) &:=   \sin a \, b'' +  2 \cos a\, a'b'   + Va' +\alpha V  \sin a\, b'    - F_1^\Lambda(a,b),\nonumber \\
G_2(u,w,V;\Lambda)&:= a'' - \half  \sin 2a\, {b'}^2   + \alpha V a' -V\sin a\, b'   - F_2^\Lambda(a,b),
\end{align}
and
\begin{align}
 F_1^\Lambda(a,b) &
:= -\uv{p} \cdot \nabla_{\uv{m}} U^\Lambda(\uv{m}(a,b)), 
\nonumber \\
  F_2^\Lambda(a,b) &
 := \uv{n} \cdot \nabla_{\uv{m}} U^\Lambda(\uv{m}(a,b)).
\label{eq: F1 and F2}
\end{align}
Here, $\uv{n} := (\cos a \cos b, -\sin a, \cos a \sin b)$ and $\uv{p} := \uv{m} \times \uv{n}$ form an orthogonal basis for the plane normal to $\uv{m}$. 
The boundary conditions \eqref{BC} imply that
\beq \label{eq: F1 = F2 = 0}
F_1^\Lambda(\pola_\pm^\Lambda, \azim_\pm^\Lambda)  = F_2^\Lambda(\pola_\pm^\Lambda, \azim_\pm^\Lambda)  = 0,
\eeq
which expresses the vanishing of torques on $\uv{m}(\xi)$ as $\xi$ approaches $\pm \infty$.

It is straightforward to show that $a',b',a''$ and $b''$ belong to $L^2(\R)$ while  $a,b,a',b'$ belong to $L^\infty(\R)$.  It follows that the terms in $G_1$ and $G_2$ involving derivatives of $a$ and $b$ belong to 
$L^2(\R)$.  From \eqref{ReferenceConfiguration}, it is straightforward to show that $F^\Lambda_1$ and $F^\Lambda_2$ also belong to $L^2(\R)$.  For example, to show that $F^\Lambda_1$ is square-integrable on $(0,\infty)$, we note that for $\xi > 0$, $(a,b)$ can be expressed as $(\pola_+^\Lambda + s, \azim_+^\Lambda + t)$, where 
\beq \label{eq: s and t}
s = (\psi_* - \psi_+^{\Lambda_*}) -  (1- \theta) (\psi_+^\Lambda - \psi_+^{\Lambda_*}) + u, \quad 
t = (\beta_* - \beta_+^{\Lambda_*}) -  (1- \theta) (\beta_+^\Lambda - \beta_+^{\Lambda_*}) + w.
\eeq
It is evident that $s,t \in  L^2((0,\infty))$.  Therefore, since $F_1^\Lambda(\pola_+^\Lambda, \azim_+^\Lambda,)  = 0$, it follows from the Mean Value Theorem and boundedness of derivatives of $F^\Lambda_1$ that $F^\Lambda_1(a,b) = F^\Lambda_1(\pola_+^\Lambda+s(\xi),\azim_+^\Lambda + t(\xi))$ is bounded by $|s(\xi)| + |t(\xi)|$ up to a multiplicative constant.  A similar argument, with $(\pola_+^\Lambda, \azim_+^\Lambda)$ replaced by $(\pola_-^\Lambda, \azim_-^\Lambda)$, shows that  $F^\Lambda_1$ is square-integrable on $(-\infty,0))$, and likewise for $F^\Lambda_2$.

Finally, just as solutions $\uv{m}$ of \eqref{TWEq} are determined up to translations, so too are solutions $(u,w)$ of $G_1 = G_2 = 0$.  
This degeneracy can be lifted by choosing the translate which is closest to the static solution $\uv{m}_*$ in the $L^2(\R)$-norm.  This is equivalent to the 
condition $g=0$,
where
\beq\label{g}
g(u,w,V;\Lambda) := \avg{b, \azim_*'}_{L^2(\R)}.
\eeq

With these considerations, we can formulate the TW equation in terms of a map between Banach spaces.  We define 
 \beq
\label{XYZ}
X=H^2(\R)\times H^2(\R)\times \R, \ \ Y = \R^4 = \{\Lambda\}, \ \  Z= L^2(\R)\times L^2(\R)\times\R,
\eeq
and define

\beq \v{G}: X \times Y  \to Z; \ \ 
(u,w,V; \Lambda) \mapsto (G_1, G_2, g)(u,w,V; \Lambda). \label{Map} 
\eeq
Then $\uv{m}(a,b)$
satisfies \eqref{TWEq} if and only if $\v{G}(u,w,V;\Lambda) = 0$.
In order to prove existence of solutions by applying the Implicit Function Theorem, we shall need the following result:
\begin{proposition}\label{prop: Differentiability}
$\mathbf G$ is continuously Fr\'echet differentiable.
\begin{proof} 
The arguments are standard, but because of the large number of terms contained in the Fr\'echet derivative of $\v{G}$,  we restrict the discussion to a few representative terms.   
 
The functional derivative with respect to $w$ of the first term in $G_1$, namely $T := \sin a \,b''$, is  the linear operator $B: H^2(\R) \rightarrow L^2(\R); \phi \mapsto  \sin a \,\phi''$.  This is clearly bounded, and depends continuously on $u$, $w$ and $\Lambda$.  The functional derivative of $T$ with respect to $u$ is the linear operator $A: \chi \mapsto b'' \cos a \, \chi$.  We note that $b'' \in L^2(\R)$ while $\chi, \cos a \in L^\infty(\R)$, so that $b'' \cos a \, \chi \in L^2(\R)$.   Also, $||\chi||_{L^\infty}$ is bounded by $||\chi||_{H^2}$.  It follows that $B$ is a bounded map from $H^2(\R)$ to $L^2(\R)$ and is continuous in its arguments.  
The  functional derivatives of the remaining terms with respect to $u$ and $w$ are treated similarly; in general, they can be expressed as linear maps $f \mapsto h$, where $h$ is given by the product of an $L^2$-function and an $L^\infty$-function, so that $h \in L^2(\R)$.  Moreover, the $L^2$-norm of $h$ can be bounded by the $H^2$-norm of $f$.

Establishing continuous Fr\'echet differentiability with respect to $\Lambda$ requires the vanishing-torque condition \eqref{eq: F1 = F2 = 0}.  Consider, for example,  the term $F^\Lambda_1$ as given by \eqref{eq: F1 and F2}. It is necessary to show that 
$D_\Lambda F^\Lambda_1(a,b)$ belongs to  $L^2(\R)$ and that, as an element of $L^2(\R)$, $D_\Lambda F^\Lambda_1(a,b)$ depends continuously on $u$, $w$ and $\Lambda$.
We have that
\beq 
\label{eq: D_Lambda F_1}
D_\Lambda F^\Lambda_1= \left.\left(\pd{F^\Lambda_1}{a} D_\Lambda a 
+ \pd{F^\Lambda_1}{b} D_\Lambda b  
+ \pd{F^\Lambda_1}{\Lambda}\right)\right|_{a,b,\Lambda},
\eeq 
where the notation $f|_x$ indicates the evaluation of function $f$ at the point $x$.
Consider first the interval  $\xi \in (0,\infty)$.  From \eqref{eq: s and t},  $(a,b) = (\pola_+^\Lambda + s, \azim_+^\Lambda+t)$, where
$s$, $t$, $D_\Lambda s$ and $D_\Lambda t$ all belong to $L^2((0,\infty))$.
Substituting in \eqref{eq: D_Lambda F_1}, we may write  that
\beq
D_\Lambda F^\Lambda_1 = Q + \left. \left(\pd{F^\Lambda_1}{a} D_\Lambda s  \ + \pd{F^\Lambda_1}{b} D_\Lambda t\right)\right|_{a,b},  
\eeq 
where
\beq
Q = \left.\left( \pd{F^\Lambda_1}{a} D_\Lambda \pola_+^\Lambda  + \pd{F^\Lambda_1}{b} D_\Lambda \azim_+^\Lambda + \pd{F^\Lambda_1}{\Lambda}\right)\right|_{a,b}.
\eeq
Since $\partial F^\Lambda_1/\partial a$ and $\partial F^\Lambda_1/\partial b$ are bounded, the terms $(\partial F^\Lambda_1/\partial a) D_\Lambda s$ and $(\partial F^\Lambda_1/\partial b )D_\Lambda t$ belong to $L^2((0,\infty))$. 
As for $Q$, we note that
\beq
\left. Q \right|_{\pola_+^\Lambda ,\azim_+^\Lambda,\Lambda} = 
\left. \left(\pd{F^\Lambda_1}{a} D_\Lambda \pola_+^\Lambda  + \pd{F^\Lambda_1}{b} D_\Lambda \azim_+^\Lambda + \pd{F^\Lambda_1}{\Lambda}\right)\right|_{\pola_+^\Lambda ,\azim_+^\Lambda,\Lambda} = 0,
\eeq
which follows from differentiating the vanishing-torque condition \eqref{eq: F1 = F2 = 0} with respect to $\Lambda$.  Using the Mean Value Theorem, we may write that
\beq
\pd{F^\Lambda_1}{a}(a,b) = \pd{F^\Lambda_1}{a}(\pola_+^\Lambda ,\azim_+^\Lambda) +  \frac{\partial^2 F^\Lambda_1}{\partial a^2}(\tilde a, \tilde b)
s +   \frac{\partial^2 F^\Lambda_1}{\partial a \partial b}(\hat a, \hat b, \Lambda) t,
\eeq
where  $\pola_+^\Lambda \le  \tilde a, \hat a \le \pola_+^\Lambda + s$ and $ \azim_+^\Lambda \le  \tilde b, \hat b \le \azim_+^\Lambda + t$.
Using similar expressions for the remaining terms in $Q$, we may write that 
\begin{multline}\label{eq: preceding expression}
Q = \left.\brackets{\pdd{F^\Lambda_1}{a}D_\Lambda \pola_+^\Lambda 
+ \frac{\partial^2 F^\Lambda_1}{\partial a \partial b} D_\Lambda \azim_+^\Lambda
+\frac{\partial^2 F^\Lambda_1}{\partial a \partial \Lambda}} \right|_{\tilde a, \tilde b}  s \  + \\ 
+
\left.\brackets{\frac{\partial^2 F^\Lambda_1}{\partial a \partial b} D_\Lambda \pola_+^\Lambda 
+ \pdd{F^\Lambda_1}{b}  D_\Lambda \azim_+^\Lambda
+\frac{\partial^2 F^\Lambda_1}{\partial b \partial \Lambda}} \right|_{\hat a, \hat b} t.
\end{multline}
As $s$ and $t$ belong to $L^2((0,\infty))$ and their coefficients in  \eqref{eq: preceding expression} are bounded, it follows that $Q \in L^2((0,\infty))$.

An analogous argument shows that for $\xi < 0$, $D_\Lambda F^\Lambda_1(a,b) \in L^2((-\infty,0))$.  It follows that 
$D_\Lambda F^\Lambda_1(a,b) \in L^2(\R)$, and is  continuous with respect to $u$, $w$, $\Lambda$. One may extend similar arguments to all other terms, and the claim follows.
 \end{proof}
\end{proposition}

%%%%%%%%%%%%%%%%%%%%%%%%%%%%%%%%%%%%%%%%

\section{Existence of travelling waves}

The existence of travelling waves is given by Theorem~\ref{TWThm1} in the anisotropy-dominated case, and by Theorem~\ref{TWThm2} in the transverse-field dominated case.  Properties of certain one-dimensional Schr\"odinger operators relevant to the analysis are proved in Lemmas~\ref{LBoundedBelow}, \ref{LNonNegative} and \ref{lem: ker N = 0}. 

\begin{theorem}[Anisotropy-dominated TWs]
\label{TWThm1} For all parameters $\Lambda = (H_1, H_2, H_3, K_2)$ sufficiently close to $\Lambda_W $, there exists  a solution $\uv{m}(\xi;\Lambda)$ of  the TW equation \eqref{TWEq} with velocity $V(\Lambda)$. Both $\uv{m}$ and $V$ are continuously differentiable in $\Lambda$, and  for $\Lambda = \Lambda_W$, $\uv{m}$ coincides with static profile $\uv{m}_W$.  
\end{theorem} 

\begin{proof}
We recall the Implicit Function Theorem (see e.g. \cite{Hunter}): Given Banach spaces $X, Y, Z$ and a continuously Fr\'echet differentiable map $\v{G}:X \times Y \to Z$ such that  
$\v{G}(x_0,y_0)=0$, if the linear operator $D_x\v{G}(x_0,y_0): X\rightarrow Z$ has bounded inverse, then  there exist open balls $B_X = B_r(x_0)\subset X$ and $B_Y = B_s(y_0)\subset Y$ such that for all $y\in B_Y$, there exists a unique $x\in B_X$ satisfying $\v{G}(x,y)=0$. Moreover, the implicit function $f: B_Y\rightarrow B_X$ defined by $\v{G}(f(y),y) = 0$  is continuously differentiable on $B_Y$.

We take $X, Y, Z$ as in \eqref{XYZ} and  $\v{G}$ as in \eqref{Map}, with $\Lambda_* = \Lambda_W = (0,0,0,K_2)$.  From Proposition~\ref{prop: Differentiability}, $\v{G}$ is continuously Fr\'echet differentiable. We let  $x_0 := (u_0, w_0, V_0) = (0,0,0)$ and $y_0 = \Lambda_W$, so that $\v{G}(x_0,y_0) = 0$.
We let
\beq
\Dop_W :=D_{(u,w,V)}\v{G}(0,0,0;\Lambda_W): X \to Z.
\eeq 
Then for $(f_1, f_2,\mu )\in X =  H^2(\R) \times H^2(\R)\times \R$,  we have that
\beq
\mathcal{D}(f_1,f_2,\mu) = (-Lf_2 + \alpha \mu \azim_W',  - (L+K_2)f_1 - \mu \azim_W', \avg{\azim_W',f_2}),
\label{DOperator}
\eeq
where  $L$ is the Schr\"odinger operator given by
\beq
L:= - \ddbydd{}{\xi} + W(\xi), \qquad W(\xi)= \cos 2 \azim_W =  \frac{\azim_W'''}{\azim_W'} = 1 - 2\sech^2 \xi.
\label{LDefinition}
\eeq
Clearly
\beq \label{eq: L beta' = 0}
L\, \azim_W' = 0, 
\eeq
so that $\ker L$ is spanned by $\azim_W'$.  (Note that the kernel of a Schr\"odinger operator is at most one dimensional, since its eigenfunctions  are necessarily nondegenerate.)

First, we show that $\Dop_W$ is bounded away from zero.  
We write $f_j = f_j^\perp + c_j \beta_W'$, where $f_j^\perp$ is orthogonal to $\beta'_W$.  
From \eqref{DOperator} and \eqref{eq: L beta' = 0}, it follows that
\beq \label{eq: quad form}  \norm{\Dop_W(f_1,f_2,\mu)}_{Z} = || (L+K_2)f_1^\perp||^2_{L^2} + || L f_2^\perp||^2_{L^2}  + 
 ((K_2 c_1 + \mu)^2  + \alpha^2 \mu^2 + c_2^2) ||\beta_W'||_{L^2}.\eeq
For $\alpha \neq 0$, the quantity$(K_2 c_1 + \mu)^2 + \alpha^2 \mu^2$, regarded as a quadratic form  in $c_1$  and $\mu$  is positive definite, so that there exists $C > 0$ such that
\[ ||\beta_W'||_{L^2} ((K_2 c_1 + \mu)^2 + c_2^2 + \alpha^2 \mu^2) \ge C \left ( || c_1 \beta_W' ||^2_{H^2}  + || c_2 \beta_W' ||^2_{H^2}+ \mu^2\right).\]

From \eqref{LDefinition}, it is clear that i) $W$ is smooth, ii) $\lim_{\xi \rightarrow \pm \infty} W(\xi) =  1$,  and iii) $\lim_{\xi \rightarrow \pm \infty} W''(\xi) = 0$.
In 
Lemma~\ref{LBoundedBelow}  below, it is shown that if the potential $W$ in a one-dimensional Schr\"odinger operator has these properties, then it is bounded away from zero on the orthogonal complement of its kernel, which for $L$ is the subspace orthogonal to $\azim_W'$.
In Lemma~\ref{LNonNegative} it is shown that $L+K_2$ has trivial kernel for $K_2 >0$.  Together with  Lemma~\ref{LBoundedBelow}, this implies that $L + K_2$ is bounded away from zero.  Therefore, there exists
$C' > 0$ such that
\[ || (L+K_2)f_1^\perp||^2_{L^2} \ge C' || f_1^\perp||^2_{H^2}, \quad ||  L f_2^\perp||^2_{L^2} \ge C' || f_2^\perp||^2_{H^2}.\] 
Taking $C'' = \min (C, C')$, we obtain 
\beq \label{eq: Dop bounded away from zero}
\norm{\Dop_W(f_1,f_2,\mu)}_{Z} \geq C'' \norm{(f_1,f_2,\mu)}_{X}.
\eeq

Next, we establish that $\Dop_W:X\rightarrow Z$ is onto.
The fact that $\Dop_W$ is bounded away from implies that $ \Dop_W$ has closed range $R(\Dop_W)$, so  it suffices to show that $R(\Dop_W)^\perp = 0$. 

Suppose that $(\phi_1, \phi_2, r) \in R(\Dop_W)^\perp \subset Z = L^2(\R)\times L^2(\R) \times \R$, so that
 \beq
\avg{(\phi_1,\phi_2, r), \Dop_W(f_1,f_2,\mu)}_Z = 0\ \   \text{for all  $(f_1,f_2,\mu) \in X$}. \label{Ontocondition}
\eeq

Writing out the terms in \eqref{Ontocondition}, we get that
\beq \label{Ontocondition2}
-\avg{\phi_1,Lf_2 - \alpha\mu \beta_W'}_{L^2} 
- \avg{\phi_2,(L+K_2)f_1 + \mu \beta_W'}_{L^2}  + r\avg{\beta_W', f_2}_{L^2} = 0.
\eeq

Taking  $f_2=0$, $\mu =0$, we have for all $f_1 \in H^2(\R)$ that
$$
\avg{\phi_2,(L+K_2)f_1}_{L^2} = 0.
$$
Regarding $L$ as an unbounded operator on $L^2(\R)$ with domain $D(L)$, the preceding implies that $\phi_2 \in D(L^\dag)$.  Since $L$ is self adjoint,  it follows that
\[ 0 = \avg{\phi_2,(L+K_2)f_1}_{L^2} = \avg{f_1, ( L^\dag+K_2)\phi_2}_{L^2} =  \avg{f_1, ( L+K_2)\phi_2}_{L^2}.\]
Thus, $(L+K_2)\phi_2= 0$.   Lemma~\ref{LNonNegative}  implies that $L+K_2$ has trivial kernel, so that $\phi_2 = 0$.

Next, we take $f_1=0$, $\mu=0$ in \eqref{Ontocondition2}  to get that
$$
\avg{\phi_1,Lf_2}_{L^2} = r\avg{\beta_W', f_2}_{L^2}
$$
for all $f_2\in H^2(\R)$.  Arguing as above, it follows that 
$\phi_1 \in D(L^\dag)$, so that
$ \avg{L\phi_1,f_2}_{L^2} = r\avg{\beta_W', f_2}_{L^2}$ for all $f_2\in H^2(\R)$, which implies that
$$
 L\phi_1 = r\beta_W'.
$$
Taking inner products with $\beta'_W$, we get that $r=0$ and that $\phi_1 \in \ker L = \text{span}\, \{\beta_W'\}$. Finally, taking $f_1=f_2=0$ in \eqref{Ontocondition2},  we see that
\beq
\avg{\phi_1,\beta_W'}_{L^2} =0,
\label{orthcond}
\eeq
which implies that $\phi_1 = 0$.

It follows that $\Dop_W$ is onto, and therefore invertible.  Since $\Dop_W$ is bounded away from zero,  $\Dop_W^{-1}$ is bounded.  The assertions in Theorem~\ref{TWThm1}
then follow   from the Implicit Function Theorem.
\end{proof}

\begin{remark}  The potential $W$ in Eq.~\eqref{LDefinition} is a particular case of the modified  P\"oschl--Teller potential \cite{Morse}, whose spectral properties are known rather explicitly.  We shall not make use of these explicit results, however, but instead  give a self-contained analysis in Lemmas~\ref{LBoundedBelow} and \ref{LNonNegative}. This has the advantage of carrying over to the case of transverse-field-dominated TWs, for which corresponding explicit results for the Schr\"odinger operator that appear there  are not available.
\end{remark}

\begin{lemma}
\label{LBoundedBelow}  Let $L = -\mathrm{d}^2/\mathrm{d}\xi^2 + W(\xi) :H^2(\R)\to L^2(\R)$ be a Schr\"odinger operator.  If i) $W \in C^2(\R)$, ii) $W$ has positive limiting values $W_\pm$ as $\xi\rightarrow \pm \infty$, and iii) $W''(\xi) \rightarrow 0$ as $\xi \rightarrow \pm \infty$, then $L$ is bounded away from zero on the orthogonal complement of $\ker L$.  That is, there exists $C> 0$ such that  
$$
\norm{Lu}_{L^2(\R)} \geq C \norm{u}_{H^2(\R)} \ \ \text{for all $u \in \ker(L)^\perp $}.
$$
\end{lemma}
\begin{proof}
Assume to the contrary that there exists a normalized sequence $u_n \in \ker(L)^\perp \cap H^2(\R)$ 
such that $Lu_n \to  0$ strongly in $L^2(\R)$. 
We can  then extract a weakly convergent subsequence (not relabelling), $u_{n} \weak u$ in $H^2(\R)$. As $L$ is bounded, it follows that $Lu_n \weak Lu$ weakly in $L^2(\R)$. 
As the weak and strong limits must coincide, it follows that $Lu=0$, i.e. $u \in \ker L$. Since  $u\in (\ker L)^\perp$ by assumption, we conclude that
 $u=0$, so that $u_n$ converges weakly to zero in $H^2(\R)$. 

Integrating by parts, we can write that
$$
\norm{Lu_n}_{L^2(\R)}^2 = \int_\R \brackets{\abs{u_n''}^2 + 2 W \abs{u_n'}^2 + (W^2-W'')\abs{u_n}^2 }.
$$
Taking  $c = \min (W_+, W_-) > 0$, we may write that 
\beq
\norm{Lu_n}_{L^2(\R)}^2 = \int_\R \brackets{\abs{u_n''}^2 + 2 c \abs{u_n'}^2 +  c^2 \abs{u_n}^2} - J_n,
\label{NormAB}
\eeq
where 
$$
J_n := \int_\R 2(c - W) \abs{u_n'}^2 + (c^2 - W^2 + W'') \abs{u_n}^2.
$$
Given $\epsilon > 0$, we can choose $\ell > 0$ such that $2(c - W) > -\epsilon$
and $c^2 - W^2 + W'' > -\epsilon$ for $|\xi| > \ell$.  Then 
\[ J_n \ge 
-\epsilon + \int_{-\ell}^\ell 2(c - W) \abs{u_n'}^2 + (c^2 - W^2 + W'') \abs{u_n}^2, 
\]
since $||u_n||_{L^2}^2 + ||u'_n||_{L^2}^2 \le 1$.
By the Rellich--Kondrachov compact embedding theorem, the fact that $u_n$ converges weakly to $0$  in $H^2(\R)$ implies that  $u_n$ and $ u_n' $ converge  strongly  to $0$ in $L^2((-\ell,\ell))$.  Therefore, 
$$
\lim_{n\rightarrow \infty} \int_{-\ell}^\ell 2(c - W) \abs{u_n'}^2 + (c^2 - W^2 + W'') \abs{u_n}^2 = 0. 
$$
It follows that  $\lim_{n\rightarrow \infty} J_n \geq -\epsilon$, and since $\epsilon$ is arbitrary, that
$ \lim_{n\rightarrow \infty} J_n \leq 0$.
From  \eqref{NormAB}, we conclude that
\beq
\lim_{n\rightarrow\infty} \norm{Lu_n}_{L^2(\R)}^2 \geq \lim_{n\rightarrow \infty} \int_\R \brackets{\abs{u_n''}^2 + 2 c \abs{u_n'}^2 +  c^2 \abs{u_n}^2} \geq \min\{1,2c,c^2\} > 0.
\eeq 
But this contradicts the fact  that $Lu_n \rightarrow 0$.
\end{proof}

\begin{lemma}\label{LNonNegative}
Let $L$ be the Schr\"odinger operator given by \eqref{LDefinition}.  Then
\[ \avg{\phi, (L+K_2)\phi}_{L^2(\R)} \geq K_2 ||\phi||_{L_2}^2.\]
\end{lemma}

\begin{proof}
We take $\phi = u \azim_W'$. Since $|\azim_W'| >  0$, it follows that $u \in H^2_\text{loc}(\R)$. Then
$$
L\phi = -u''\azim_W' - 2 u' \azim_W'',
$$
and hence
$$
\avg{\phi, L\phi}_{L^2} = -\int_\R u \brackets{u'' {\beta_W'}^2 + 2u'\beta_W' \beta_W''} = 
-\int_\R u \left( u' {\beta_W'}^2\right)' = 
\int_\R {u'}^2 {\beta_W'}^2 \ge 0,$$
where  we have used the fact, easily checked, that ${\beta'_W}^2 uu'$ vanishes at $\xi = \pm \infty$.
Therefore, $\avg{\phi, (L+K_2)\phi}_{L^2(\R)} \geq K_2 ||\phi||_{L_2}^2$. \end{proof}

%%%%%%%%%%%%%%%%%%%%%%%%%%%%%%%%%%%%%%
%\section{Transverse field-dominated travelling waves}  
\begin{theorem}[Transverse-field-dominated TWs]\label{TWThm2}
For all parameters $\Lambda = (H_1, H_2, H_3, K_2)$   sufficiently close to $\Lambda_T$, there exists  a solution $\uv{m}(\xi;\Lambda)$ of  the TW equation \eqref{TWEq} with velocity $V(\Lambda)$. Both $\uv{m}$ and $V$ are continuously differentiable  in $\Lambda$, and  for $\Lambda = \Lambda_T$, $\uv{m}$ coincides with the static profile $\uv{m}_T$.  
\end{theorem} 

\begin{proof}
When $\Lambda_W$ is replaced by $\Lambda_T$, 
$\Dop_T :=D_{(u,w,V)}\v{G}(0,0,0;\Lambda_T)$ is given by
\beq
\label{DopField}
\Dop_T(f_1,f_2,\mu) = (-M f_2 +\alpha\mu\azim_T' ,- N f_1 - \mu \azim_T', \avg{\azim_T',f_2}).
\eeq
Here, $\azim_T(\xi) \in H^2(\R)$  satisfies (cf  \eqref{StaticZero} and \eqref{eq: StaticZero bcs})
\beq \label{eq: beta_T'}  \azim_T' = H_3 - \sin \azim_T, \quad \lim_{\xi\rightarrow-\infty} \azim_T = \pi - \sin^{-1}(H_3), \ \  \lim_{\xi\rightarrow\infty} \azim_T =\sin^{-1}(H_3),\eeq
and $M$, $N$ are Schr\"odinger operators given by 
\beq \label{eq: MN}
M:=  -\ddbydd{}{\xi} + \frac{\azim_T'''}{\azim_T'}, \qquad  N:= -\ddbydd{}{\xi} + \frac{(\cos \azim_T)''}{\cos \azim_T} + 1.
\eeq
Clearly 
\beq \label{eq: M beta = 0} M\, \beta'_T = 0,\eeq so that $\beta'_T$ spans  $\ker M$.

Comparing \eqref{DOperator} and \eqref{DopField}, we see that $\Dop_T$ is obtained from $\Dop_W$ by making the replacements $\beta'_W \rightarrow \beta_T'$, $L \rightarrow M$ and $L+K_2 \rightarrow N$. The proof of Theorem~\ref{TWThm1}  carries over directly provided we  show that $M$ is bounded away from zero on $\ker(M)^\perp$, 
and that $N$ is bounded away from zero.

Firstly, Lemma~\ref{LBoundedBelow} implies directly that $M = -\mathrm{d}^2/\mathrm{d}\xi^2 + W$ is bounded away from zero on $\ker(M)^\perp$.
Secondly, with  $W$ taken  to be $(\cos \azim_T)''/ \cos \azim_T + 1$, the fact that $N = -\mathrm{d}^2/\mathrm{d}\xi^2 + W$ is bounded away from zero follows from 
 Lemma~\ref{LBoundedBelow} (as  \eqref{eq: beta_T'} implies that  $W$ is smooth, $\lim_{\xi \rightarrow \pm \infty} W(\xi) = 1$,  and $\lim_{\xi \rightarrow \pm \infty} W''(\xi) = 0$), combined with Lemma~\ref{lem: ker N = 0} below, which implies that $N$ has trivial kernel.\end{proof}

\begin{lemma}\label{lem: ker N = 0}
Let $N$ be the Schr\"odinger operator given by~\eqref{eq: MN}.  Then
\[\avg{\phi, N\phi}_{L^2(\R)} \geq H_3^2 ||\phi||_{L_2}^2.\]
\end{lemma}
\begin{proof}
From \eqref{DOperator} and \eqref{DopField}, we get that
\[ N = -\frac{d^2}{ d\xi^2} + \cos 2\beta_T + 3H_3\sin \beta_T - H_3^2, \quad (\beta_T')'' = \left(\cos 2\beta_T + H_3 \sin \beta_T\right)\beta_T'.\]

Therefore, 
\[ N\beta_T' = H_3(2\sin \beta_T - H_3) \beta_T'.\]
We let 
$\phi = u \beta_T'$.
Since $|\beta_T'| < 0$, it follows that $u \in H^2_{\text{loc}}(\R)$.  Calculation gives
\[ N\phi = -u'' \beta_T' - 2u'\beta_T'' + H_3 (2\sin \beta_T - H_3) \phi,\]
and hence
\begin{multline*} \ip{\phi}{N\phi}_{L^2(\R)} = -\int_\R u \left( u'' {\beta_T'^2} + 2u'\beta_T' \beta_T''\right) + H_3 \int_\R (2\sin \beta_T - H_3) \phi^2 = \\
=\int_\R {u'}^2 {\beta_T'^2} +  H_3 \int_\R (2\sin \beta_T - H_3) \phi^2 \geq H_3^2 ||\phi||^2_{L^2},\end{multline*}
where in the second equality  we have used the fact (easily checked)  that ${\beta'_T}^2 uu'$ vanishes at $\xi = \pm \infty$, while 
in the last equality we have  
$\sin \beta_T\geq H_3$.
\end{proof}

\section{Conclusion}
We have proven existence of travelling-wave solutions to the Landau--Lifshitz--Gilbert equation in a thin ferromagnetic nanowire. This was accomplished by virtue of the Implicit Function Theorem, which provides local existence and uniqueness (up to translation) of travelling-wave profiles close to static energy minimizers satisfying tail-to-tail boundary conditions (thus representing propagating domain walls), with velocities in a neighbourhood of 0. We proved the existence of solutions of this type for values of the physical parameters close to two ranges in the parameter space: the first with zero applied field $\v{H}_a = 0$ and nonzero transverse anisotropy $K_2>0$, and the second with nonvanishing transverse applied field  $\v{H}_a=H_2 \yhat +H_3\zhat$, $0<H_2^2+H_3^2<1$, and vanishing  hard axis anistropy $K_2=0$.

%%%%%%%%%%%%%%%%%%%%%%%%%%%%%%%%%%%%%%%%%%%%%%%%%%%%%%%%%%%%%%%

\section*{Acknowledgments} RGL would like to acknowledge the support of an EPSRC doctoral training award granted by University of Bristol from 2010-2014. JMR and VS thank the EPSRC for support under grant  EP/K02390X/1.

\label{lastpage}
\end{document}